\documentclass[10pt,twoside,a4paper]{article}

\usepackage[utf8]{inputenc}

\usepackage{amsmath}
\numberwithin{equation}{section}

\usepackage{amssymb}

\usepackage{mathrsfs}

\usepackage{microtype}

\usepackage{tikz}
\usepackage{tikz-3dplot}
\usetikzlibrary{arrows,shapes}
\usetikzlibrary{fadings}
\usetikzlibrary{intersections}
\usetikzlibrary{decorations.pathreplacing}

\usepackage[backend=bibtex,maxnames=10,backref=false,hyperref=true,safeinputenc]{biblatex}
\DeclareFieldFormat[report]{title}{``#1''}
\DeclareFieldFormat[book]{title}{``#1''}

\newcommand{\ourtitle}{Quantitative Photoacoustic Imaging in the Acoustic Regime using SPIM}

\title{\ourtitle}
\author{%
Alexander Beigl$^2$\\{\footnotesize\href{mailto:alexander.beigl@oeaw.ac.at}{alexander.beigl@oeaw.ac.at}}%
\and Peter Elbau$^1$\\{\footnotesize\href{mailto:peter.elbau@univie.ac.at}{peter.elbau@univie.ac.at}}%
\and Kamran Sadiq$^2$\\{\footnotesize\href{mailto:kamran.sadiq@oeaw.ac.at}{kamran.sadiq@oeaw.ac.at}}%
\and Otmar Scherzer$^{1,2}$\\{\footnotesize\href{mailto:otmar.scherzer@univie.ac.at}{otmar.scherzer@univie.ac.at}}%
}

\usepackage[pdftex,colorlinks=true,linkcolor=blue,citecolor=green,urlcolor=blue,bookmarks=true,bookmarksnumbered=true]{hyperref}
\AtBeginDocument{
	\hypersetup
	{
	    pdfauthor={Alexander Beigl, Peter Elbau, Kamran Sadiq, Otmar Scherzer},
	    pdfsubject={Quantitative Photoacoustic Imaging in the Acoustic Regime using SPIM},
	    pdftitle={\ourtitle},
	    pdfkeywords={Photoacoustic Imaging}
	}
}

\usepackage[hyperref,amsmath,thmmarks]{ntheorem}
\usepackage{aliascnt}

\newtheorem{lemma}{Lemma}[section]

\newaliascnt{proposition}{lemma}
\newtheorem{proposition}[proposition]{Proposition}
\aliascntresetthe{proposition}

\newaliascnt{corollary}{lemma}
\newtheorem{corollary}[corollary]{Corollary}
\aliascntresetthe{corollary}

\newaliascnt{theorem}{lemma}

\aliascntresetthe{theorem}

\theorembodyfont{\normalfont}
\newaliascnt{definition}{lemma}

\aliascntresetthe{definition}

\newaliascnt{assumption}{lemma}

\aliascntresetthe{assumption}

\newaliascnt{hypothesis}{lemma}

\aliascntresetthe{hypothesis}

\newaliascnt{remark}{lemma}

\aliascntresetthe{remark}

\newaliascnt{example}{lemma}

\aliascntresetthe{example}

\newaliascnt{convention}{lemma}

\aliascntresetthe{convention}

\theoremstyle{nonumberplain}
\theoremseparator{:}
\theoremheaderfont{\normalfont\itshape}
\theoremsymbol{\ensuremath{\square}}
\newtheorem{proof}{Proof}

\usepackage[a4paper,centering,bindingoffset=0cm,marginpar=2cm,margin=1.5cm,includeheadfoot]{geometry}

\usepackage[pagestyles]{titlesec}
\titleformat{\section}[block]{\large\sc\filcenter}{\thesection.}{0.5ex}{}[]
\titleformat{\subsection}[runin]{\bf}{\thesubsection.}{0.5ex}{}[.]

\makeatletter
\AtBeginDocument{
\newpagestyle{headers}%
{%
	\headrule
	\sethead%
		[\footnotesize\thepage]%
		[\footnotesize\sc Alexander Beigl, Peter Elbau, Kamran Sadiq, Otmar Scherzer]%
		[]%
		{}%
		{\footnotesize\sc\ourtitle}%
		{\footnotesize\thepage}
	\setfoot{}{}{}
}
\pagestyle{headers}}
\makeatother

\usepackage[font=footnotesize,format=plain,labelfont=bf,textfont=up,width=0.8\textwidth,labelsep=period,justification=centerlast]{caption}

\usepackage{tikz}

\usepackage{dsfont}

\newcommand{\R}{\mathds{R}}

\renewcommand{\d}{\mathrm d}

\let\RE\Re
\let\Re=\undefined
\DeclareMathOperator{\Re}{\RE e}
\let\IM\Im
\let\Im=\undefined
\DeclareMathOperator{\Im}{\IM m}
\DeclareMathOperator{\supp}{supp}

\DeclareMathOperator{\Div}{div}
\DeclareMathOperator{\interior}{int}

\begin{document}

\maketitle
\vspace*{-2em}
\begin{center}
\parbox[t]{0.4\textwidth}{\footnotesize
\begin{tabbing}
$^1$\=Computational Science Center\\
\>University of Vienna\\
\>Oskar-Morgenstern-Platz 1\\
\>A-1090 Vienna, Austria
\end{tabbing}}
\hfil
\parbox[t]{0.4\textwidth}{\footnotesize
\begin{tabbing}
$^2$\=Johann Radon Institute for Computational\\
\>\hspace*{1em}and Applied Mathematics (RICAM)\\
  \>Altenbergerstra{\ss}e 69\\
\>A-4040 Linz, Austria
\end{tabbing}}
\end{center}
\vspace*{2em}

\begin{abstract}
While in standard photoacoustic imaging the propagation of sound waves is modeled by the standard wave 
equation, our approach is based on a generalized wave equation with variable sound speed and material density, 
respectively. In this paper we present an approach for photoacoustic imaging, which in addition to recovering 
of the absorption density parameter, the imaging parameter of standard photoacoustics, also allows to reconstruct
the spatially varying sound speed and density, respectively, of the medium. 
We provide analytical reconstruction formulas for all three parameters based in a linearized model based on single 
plane illumination microscopy (SPIM) techniques. 
\end{abstract}

\section{Introduction}
Photoacoustic imaging (PAI) is a novel imaging technique which uses pulsed laser excitation in the 
visible or infrared frequency regime to illuminate a specimen and measures the acoustic response of 
the medium (see \cite{Wan09} and some mathematical survey references \cite{KucKun08,WanAna11,KucSch15}).
Mathematical models of standard PAI do not take into account variable sound speed and elastic 
material parameters of the medium, such as compressibility and mass density \cite{Wan09}.
More sophisticated models take into account spatially varying sound speed, but assume the sound speed to be known 
from measurements of other modalities (like for instance elastography and ultrasound 
experiments). PAI which takes into account given spatially varying sound speed and elasticity parameters 
has been considered in \cite{AgrKuc07,HriKucNgu08,Hri09,SteUhl09a,SteUhl11,QiaSteUhlZha11,KucSte12,BelGlaSch16a,BelGlaSch16b}.
Little is known, how and when it is (at least) theoretically possible, to recover in addition to the photoacoustic 
imaging parameter, the absorption density, also sound speed and elastic 
parameters. \cite{KirSch12} provided a method for parallel estimation of the sound speed and the photoacoustic 
imaging parameter via \emph{focusing to planes}, which is called \emph{single plane illumination microscopy} 
for a Born-approximation of the wave equation. This paper is the starting point 
for the present one, which shows that in addition to the sound speed, also elastic parameters can be recovered. 
Parallel sound speed and photoacoustic imaging parameter estimation has been considered as a nonlinear inverse problem 
(without making the assumption of a Born-approximation) in \cite{SteUhl13a,SteUhl13c,LiuUhl15}.

The main goal of this paper is to perform quantitative imaging in the acoustic regime: The photoacoustic imaging 
process consists of three processes, the acoustic wave propagation, the optical illumination and the visco-elastic 
part to transform optical energy into acoustic waves. Thus our topic might not be confused with quantitative imaging 
in the optical illumination part, where optical parameters of the medium are identified 
(see e.g. \cite{BalUhl10,BalRenUhlZho11,BalUhl12,BalRen11a,BalRen12,BalBonMonTri13,BalZho14}).

In the following we describe the proposed experimental setup.
We consider the specimen to be photoacoustically imaged to be supported in a bounded domain $\Omega_0\subset\R^3$ 
embedded into a known environment. To analyse the specimen's interior structure via photoacoustic imaging, it is 
illuminated with a laser pulse. This produces a local light fluence $\Phi_{r,\theta}:\R^3\to[0,\infty)$ 
(integrated over the short time interval 
of the pulse), where we want to assume that the laser can be tuned in such a way that the function $\Phi_{r,\theta}$ 
is mainly supported in a small vicinity of the plane
\[ E_{r,\theta} = \{x\in\R^3\mid x\cdot\theta = r\} \]
for some parameters $r\in[0,\infty)$ and $\theta\in \mathbf{S}^2$. Ideally, we are illuminating only single planes. 

According to the photoacoustic effect, the object will absorb at each point parts of the light depending on the 
spatially varying absorption coefficient $\mu:\R^3\to[0,\infty)$ and transform the absorbed energy first into heat, 
which is then transformed proportional to the Grüneisen parameter $\gamma:\R^3\to[0,\infty)$ into a locally varying 
pressure distribution
\[ P^{(0)}_{r,\theta}(x) = \gamma(x)\mu(x)\Phi_{r,\theta}(x). \]
To avoid getting signals from everywhere, we want to assume that the absorption is confined to some bounded domain 
$\Omega\supset\overline{\Omega_0}$, that is, $\supp\mu\subset\Omega$.

This initial pressure distribution $P^{(0)}$ will then propagate as an elastic wave through the medium. To describe 
this wave, we model the object as an elastic body with smooth bulk modulus $K:\R^3\rightarrow[0,\infty)$ and with a 
vanishing shear modulus. 
Then we know from linear elasticity theory that the stress tensor 
$\sigma_{r,\theta}\in C^2([0,\infty)\times\R^3;\R^{3\times3})$ as a function of time and space has the simple form
\[ \sigma_{r,\theta;ij}(t,x) = K(x) \Div_xu_{r,\theta}(t,x)\delta_{ij}, \]
where $u\in C^3([0,\infty)\times\R^3;\R^3)$ denotes the displacement vector of the material as a function of time 
and space. Introducing the mass density $\rho\in C^1(\R^3)$ of the material, this leads to the equation of motion:
\begin{equation}\label{eqSetDisplacementMotion}
\rho(x) \partial_{tt}u_{r,\theta;i}(t,x) = \sum_{j=1}^3 \partial_{x_j}\sigma_{r,\theta;ij}(t,x) = \partial_{x_i}(K(x)\Div_xu_{r,\theta}(t,x)).
\end{equation}
Moreover, the relation between the displacement vector $u_{r,\theta}$ and the pressure distribution $P_{r,\theta}\in C^2([0,\infty)\times\R^3)$ is given by
\begin{equation}\label{eqSetPressureDisplacement}
P_{r,\theta}(t,x) = -K(x)\Div_x u_{r,\theta}(t,x).
\end{equation}

To rewrite the equation of motion in terms of the pressure, we take the divergence of \eqref{eqSetDisplacementMotion} and obtain
\[ \rho(x)\partial_{tt}\Div_x u_{r,\theta}(t,x)+\nabla\rho(x)\cdot\partial_{tt}u_{r,\theta}(t,x) = \Delta_x(K(x)\Div_xu_{r,\theta}(t,x)). \]
Using again the equation \eqref{eqSetDisplacementMotion}, we can rewrite the term $\partial_{tt}u_{r,\theta}$ on the left hand side as an expression containing $u_{r,\theta}$ only in the form of $\Div_xu_{r,\theta}$ and find with the relation \eqref{eqSetPressureDisplacement} the equation of motion for the pressure $P_{r,\theta}$:
\[ \frac{\rho(x)}{K(x)}\partial_{tt}P_{r,\theta}(t,x)+\frac{\nabla\rho(x)}{\rho(x)}\cdot\nabla_x P_{r,\theta}(t,x) = \Delta_x P_{r,\theta}(t,x). \]
Introducing the parameter $\alpha(x)=\frac{K(x)}{\rho(x)}$ corresponding to the square of the speed of sound, this can be written in the form
\begin{subequations}\label{eqSetPressureIVP}
\begin{gather}
\partial_{tt}P_{r,\theta}(t,x) = \rho(x)\alpha(x)\Div_x\left(\frac1{\rho(x)}\nabla_xP_{r,\theta}(t,x)\right). \label{eqSetPressureMotion} \\
\intertext{We complement this equation with the initial conditions that at $t=0$, the pressure distribution equals the one generated by the photoacoustic effect and that there is no initial motion (this corresponds to the assumption that the absorption and the transformation of heat energy into pressure happens instantaneously):}
\begin{aligned}
\partial_tP_{r,\theta}(0,x)&=0, \\
P_{r,\theta}(0,x)&=P^{(0)}_{r,\theta}(x).
\end{aligned}
\end{gather}
\end{subequations}

Our aim is to reconstruct the product $\gamma\mu$ of absorption coefficient and Grüneisen parameter, the speed of 
sound $\sqrt\alpha$, and the mass density $\rho$ from measurements of this acoustic wave $P_{r,\theta}$. 
To accomplish this, we assume that we use the measurements
\[ m_{r,\theta}(t,\xi) = P_{r,\theta}(t,\xi)\]
for all points $\xi\in\partial\Omega$, all times $t\in[0,\infty)$, and all choices of illumination planes, 
that is, for all $r\in[0,\infty)$ and $\theta\in \mathbf{S}^2$.

In \autoref{sec:FOA} we formulate the forward problem, taking into account small perturbations of the acoustic 
parameters. We consider an asymptotic expansion of the solution of the acoustic wave equation induced by the 
deviation of sound speed and density around known constant functions. Proposition \ref{thFirSolWave} combines the 
zeroth and first order term of the expansion as an integral operator 
acting on the initial acoustic pressure wave. The kernel of this operator is of utmost importance as it contains 
information about the seeked acoustic parameters. 
In \autoref{sec:kernel} we exploit asymptotic behaviors of this kernel which will be the
main tools for the final reconstruction of the parameters and the source. In the final two \autoref{sec:dp} and \autoref{sec:reconstruction} 
we consider the idealized forward problem where we are omitting the higher order terms of the expansion derived in \autoref{sec:FOA}, which is reasonable
when the distortions of the acoustic parameters become sufficiently small. The SPIM method then allows us to extract the product of the kernel and the source by applying the inverse Radon transform. By combining various measurement data, we are able to reconstruct the source and the distortions of the sound speed and density. 

\section{First Order Approximation}
\label{sec:FOA}

To simplify the problem, we will only consider the equation of motion \eqref{eqSetPressureMotion} in the case where $\rho$ and $\alpha$ are small perturbations of constant functions. To formulate this more precisely, let us introduce an artificial parameter $\varepsilon>0$ and assume that we have a family of media with parameters
\begin{align}
\rho(x;\varepsilon) &= \rho_0+\varepsilon\rho_1(x), \label{eqFirMassDensity}\\
\alpha(x;\varepsilon) &= \alpha_0+\varepsilon \alpha_1(x), \label{eqFirSpeed}
\end{align}
with $0<\rho_{\varepsilon,\text{min}}\leq\rho(x;\varepsilon)\leq\rho_{\varepsilon,\text{max}}$ and $0<\alpha_{\varepsilon,\text{min}}\leq\alpha(x;\varepsilon)\leq\alpha_{\varepsilon,\text{max}}$.  The perturbations will be considered sufficiently smooth and supported in the domain of absorption, that is, $\supp(\rho_0-\rho)\subset\Omega$ and $\supp(\alpha_0-\alpha)\subset\Omega$.

Then, we can also expand the solution $P_{r,\theta}(\cdot,\cdot;\varepsilon)$ of the initial value problem \eqref{eqSetPressureIVP} in $\varepsilon$:
\[ P_{r,\theta}(t,x;\varepsilon) = P_{r,\theta,0}(t,x)+\varepsilon P_{r,\theta,1}(t,x)+\varepsilon^2 P_{r,\theta,2}(t,x;\varepsilon), \]
where the zeroth and first order terms fulfil
\begin{subequations}\label{eqFirP0}
\begin{align}
\partial_{tt}P_{r,\theta,0}(t,x) &= \alpha_0 \Delta P_{r,\theta,0}(t,x), \\
\partial_t P_{r,\theta,0}(0,x) &= 0, \\
P_{r,\theta,0}(0,x) &= P_{r,\theta}^{(0)}(x),
\end{align}
\end{subequations}
and
\begin{subequations}\label{eqFirP1}
\begin{align}
\partial_{tt}P_{r,\theta,1}(t,x) &= \alpha_0 \Delta P_{r,\theta,1}(t,x)+\alpha_1(x)\Delta P_{r,\theta,0}(t,x) - \frac{\alpha_0}{\rho_0}\nabla\rho_1(x)\cdot\nabla P_{r,\theta,0}(t,x), \\
\partial_t P_{r,\theta,1}(0,x) &= 0, \\
P_{r,\theta,1}(0,x) &= 0.
\end{align}
\end{subequations}

The higher order term satisfies the equation
\begin{subequations}\label{eqFirP2}
\begin{align}
\partial_{tt}P_{r,\theta,2}(t,x;\varepsilon) &= \rho(x;\varepsilon)\alpha(x;\varepsilon)\Div_x\left(\frac1{\rho(x;\varepsilon)}\nabla_xP_{r,\theta,2}(t,x;\varepsilon)\right) + F(t,x;\varepsilon), \\
\partial_t P_{r,\theta,2}(0,x;\varepsilon) &= 0, \\
P_{r,\theta,2}(0,x;\varepsilon) &= 0,
\end{align}
\end{subequations}
where
\[
F(t,x;\varepsilon)=\frac{1}{\rho(x;\varepsilon)}\frac{\rho_1(x)\alpha_0-\rho_0\alpha_1(x)}{\rho_0}\nabla\rho_1(x)\cdot\nabla P_{r,\theta,0}(t,x) + \alpha_1(x)\Delta P_{r,\theta,1}(t,x)-\frac{\alpha(x;\varepsilon)}{\rho(x;\varepsilon)}\nabla\rho_1(x)\cdot \nabla P_{r,\theta,1}(t,x) 
\]
sheds light on the coupling with the zeroth and first order expansion. The parameters $\rho_0$ and $\alpha_0$ are assumed to be known and we normalise them to $\rho_0=1$ and $\alpha_0=1$ henceforth.
The initial value problems \eqref{eqFirP0} and \eqref{eqFirP1} can be explicitly solved, see for example \cite[Section 2.4, Theorem 2]{Eva98}, which gives us the following representation for $P_{r,\theta}$.
\begin{proposition}\label{thFirSolWave}
Let $\alpha_1$, $\rho_1$ and $P_{r,\theta}^{(0)}$ be smooth functions. Then, the solution $P_{r,\theta}(\cdot,\cdot;\varepsilon)$ of the initial value problem \eqref{eqSetPressureIVP} with the parameters $\rho$ and $\alpha$ given by \eqref{eqFirMassDensity} and \eqref{eqFirSpeed} has the form
\[ P_{r,\theta}(t,x;\varepsilon) = \frac{\partial^4}{\partial t^4}\left(\int_{B_t(x)}K(t,x,y;\varepsilon)P_{r,\theta}^{(0)}(y)\d y\right)+\varepsilon^2 P_{r,\theta,2}(t,x;\varepsilon), \]
where the kernel $K$ is given by
\begin{equation}\label{eqFirSolWaveKernel}
\begin{split}
K(t,x,y;\varepsilon) &= \frac{(t-|y-x|)^2}{8\pi|y-x|}(1-\varepsilon \alpha_1(y))+\int_{\mathscr E_t(x,y)}\frac1{16\pi^2|z-x||z-y|}\varepsilon \alpha_1(z)\d z \\
&\qquad+\int_{\mathscr E_t(x,y)}\frac{t-|z-x|-|z-y|}{16\pi^2 |z-x||z-y|}\varepsilon\nabla\rho_1(z)\cdot\frac{z-y}{|z-y|}\d z \\
&\qquad+\int_{\mathscr E_t(x,y)}\frac{(t-|z-x|-|z-y|)^2}{32\pi^2 |z-x||z-y|^2}\varepsilon\nabla\rho_1(z)\cdot\frac{z-y}{|z-y|}\d z,
\end{split}
\end{equation}
where
\begin{equation}\label{eqFirSolWaveEllipsoid}
\mathscr E_t(x,y) = \{z\in\R^3\mid |z-x|+|z-y|\leq t\}
\end{equation}
is the prolate spheroid with focal points $x$ and $y$ and larger semi-axis $\frac t2$, and $P_{r,\theta,2}$ remains bounded,
\begin{equation}
 \sup_{0\leq t\leq T}\left(\|P_{r,\theta,2}(t,\cdot;\varepsilon)\|_{H^{k+1}(\mathbb{R}^3)} + \|\partial_t P_{r,\theta,2}(t,\cdot;\varepsilon)\|_{H^k(\mathbb{R}^3)}\right)
 \leq C_{k,T;\varepsilon}\int_0^T\|F(t,\cdot;\varepsilon)\|_{H^k(\mathbb{R}^3)}\d t
\end{equation}
with $C_{k,T;\varepsilon}=\mathcal{O}(1)$ as $\varepsilon\rightarrow0$.

\end{proposition}

\begin{proof}
According to \cite[Section 2.4, Theorem 2]{Eva98}, the solution $P_{r,\theta,0}$ of the initial value problem \eqref{eqFirP0} is given by
\begin{equation}\label{eqFirSolWaveP0}
P_{r,\theta,0}(t,x) = \frac{\partial}{\partial t}\left(\frac1{4\pi t}\int_{\partial B_t(x)}P_{r,\theta}^{(0)}(y)\d s(y)\right) = \frac{\partial^2}{\partial t^2}\int_{B_t(x)}\frac{P_{r,\theta}^{(0)}(y)}{4\pi |y-x|}\d y,
\end{equation}
where we used the coarea formula to write the surface integral over the sphere as the derivative of the integral over the ball.

Similarly, the solution $P_{r,\theta,1}$ of \eqref{eqFirP1} is explicitly known, see for example \cite[Section 2.4, Theorem 4]{Eva98}, and we have
\begin{align}
P_{r,\theta,1}(t,x) &= \int_0^t\frac1{4\pi(t-\tau)}\int_{\partial B_{t-\tau}(x)}\left(\alpha_1(y)\Delta P_{r,\theta,0}(\tau,y)-\nabla\rho_1(y)\cdot\nabla P_{r,\theta,0}(\tau,y)\right)\d s(y)\d\tau \nonumber \\
&= \int_{B_t(x)}\frac1{4\pi|y-x|}\left(\alpha_1(y)\Delta P_{r,\theta,0}(t-|y-x|,y)-\nabla\rho_1(y)\cdot\nabla P_{r,\theta,0}(t-|y-x|,y)\right)\d y. \label{eqFirSolWaveP1General}
\end{align}

To write this as an integral operator on $P_{r,\theta}^{(0)}$, we integrate by parts to pull all derivatives acting on $P_{r,\theta,0}$ to the other terms. 
\begin{itemize}
\item
For the first term, we use that $P_{r,\theta,0}$ is a solution of \eqref{eqFirP0} and get
\begin{align}
\int_{B_t(x)}&\frac{\alpha_1(y)\Delta P_{r,\theta,0}(t-|y-x|,y)}{4\pi|y-x|}\d y \nonumber \\
&= \frac\partial{\partial t}\left(\int_{B_t(x)}\frac{\alpha_1(y)\partial_tP_{r,\theta,0}(t-|y-x|,y)}{4\pi|y-x|}\d y\right) \nonumber \\
&= \frac{\partial^2}{\partial t^2}\left(\int_{B_t(x)}\frac{\alpha_1(y)P_{r,\theta,0}(t-|y-x|,y)}{4\pi|y-x|}\d y\right)-\frac\partial{\partial t}\left(\int_{\partial B_t(x)}\frac{\alpha_1(y)P_{r,\theta}^{(0)}(y)}{4\pi|y-x|}\d s(y)\right). \label{eqFirSolWaveP1Term1}
\end{align}
To reduce $P_{r,\theta,0}$ in the first summand to a term involving only the initial data $P_{r,\theta}^{(0)}$, we write, according to~\eqref{eqFirSolWaveP0}, 
\begin{equation}\label{eqFirSolWaveTildeP0}
P_{r,\theta,0}=\partial_{tt}\tilde P_{r,\theta,0}\quad\text{with}\quad\tilde P_{r,\theta,0}(t,x)=\int_0^t\int_0^\tau P_{r,\theta,0}(\tau_1,x)\d\tau_1\d\tau=\int_{B_t(x)}\frac{P_{r,\theta}^{(0)}(y)}{4\pi |y-x|}\d y.
\end{equation}
Then, we pull the two time derivatives outside the integral and find because of $\partial_t\tilde P_{r,\theta,0}(0,x)=\tilde P_{r,\theta,0}(0,x)=0$ for the first integral in \eqref{eqFirSolWaveP1Term1} that
\begin{equation}\label{eqFirSolWaveP1Term1a}
\int_{B_t(x)}\frac{\alpha_1(y)P_{r,\theta,0}(t-|y-x|,y)}{4\pi|y-x|}\d y = \frac{\partial^2}{\partial t^2}\left(\int_{B_t(x)}\frac{\alpha_1(y)\tilde P_{r,\theta,0}(t-|y-x|,y)}{4\pi|y-x|}\d y\right).
\end{equation}
\item
For the second term in \eqref{eqFirSolWaveP1General}, we write $P_{r,\theta,0}$ in the form \eqref{eqFirSolWaveTildeP0} and pull the time derivatives out of the integral to get
\begin{multline}\label{eqFirSolWaveP1Term2}
\int_{B_t(x)}\frac{\nabla\rho_1(y)\cdot\nabla P_{r,\theta,0}(t-|y-x|,y)}{4\pi|y-x|}\d y \\
= \frac{\partial^2}{\partial t^2}\left(\int_{B_t(x)}\frac1{16\pi^2|y-x|}\int_{B_{t-|y-x|}(y)}\nabla\rho_1(y)\cdot\frac{\nabla P_{r,\theta}^{(0)}(z)}{|z-y|}\d z\d y\right).
\end{multline}
Applying the divergence theorem to the inner integral, we obtain that
\begin{multline}\label{eqFirSolWaveP1Term2a}
\int_{B_{t-|y-x|}(y)}\nabla\rho_1(y)\cdot\frac{\nabla P_{r,\theta}^{(0)}(z)}{|z-y|}\d z \\
= \int_{\partial B_{t-|y-x|}(y)}\nabla\rho_1(y)\cdot\frac{z-y}{|z-y|}\frac{P_{r,\theta}^{(0)}(z)}{|z-y|}\d s(z)+\int_{B_{t-|y-x|}(y)}\nabla\rho_1(y)\cdot\frac{z-y}{|z-y|^3}P_{r,\theta}^{(0)}(z)\d z.
\end{multline}
\end{itemize}
We now insert \eqref{eqFirSolWaveTildeP0} into \eqref{eqFirSolWaveP1Term1a} and this into \eqref{eqFirSolWaveP1Term1}. Moreover, we combine \eqref{eqFirSolWaveP1Term2a} with \eqref{eqFirSolWaveP1Term2}. The results, we then put into \eqref{eqFirSolWaveP1General} and obtain the expression
\begin{equation}\label{eqFirSolWaveP1}
\begin{split}
P_{r,\theta,1}(t,x) &= \frac{\partial^4}{\partial t^4}\left(\int_{B_t(x)}\int_{B_{t-|y-x|}(y)}\frac{\alpha_1(y)P_{r,\theta}^{(0)}(z)}{16\pi^2|y-x||z-y|}\d z\d y\right)-\frac\partial{\partial t}\left(\int_{\partial B_t(x)}\frac{\alpha_1(y)P_{r,\theta}^{(0)}(y)}{4\pi|y-x|}\d s(y)\right) \\
&\qquad-\frac{\partial^2}{\partial t^2}\left(\int_{B_t(x)}\int_{\partial B_{t-|y-x|}(y)}\frac1{16\pi^2|y-x||z-y|}\nabla\rho_1(y)\cdot\frac{z-y}{|z-y|}P_{r,\theta}^{(0)}(z)\d s(z)\d y\right) \\
&\qquad-\frac{\partial^2}{\partial t^2}\left(\int_{B_t(x)}\int_{B_{t-|y-x|}(y)}\frac1{16\pi^2|y-x||z-y|^2}\nabla\rho_1(y)\cdot\frac{z-y}{|z-y|}P_{r,\theta}^{(0)}(z)\d z\d y\right).
\end{split}
\end{equation}

Combining the formulas \eqref{eqFirSolWaveP0} and \eqref{eqFirSolWaveP1} and replacing the surface integrals again with the coarea formula with time derivatives of integrals over the corresponding balls, we find that
\begin{equation}\label{eqFirSolWaveP0P1}
\begin{split}
P_{r,\theta,0}(t,x)+\varepsilon P_{r,\theta,1}(t,x) &= \frac{\partial^2}{\partial t^2}\left(\int_{B_t(x)}\frac{(1-\varepsilon \alpha_1(z))P_{r,\theta}^{(0)}(z)}{4\pi|z-x|}\d z\right) \\
&\qquad+\frac{\partial^4}{\partial t^4}\left(\int_{B_t(x)}\int_{B_{t-|y-x|}(y)}\frac{\varepsilon \alpha_1(y)P_{r,\theta}^{(0)}(z)}{16\pi^2|y-x||z-y|}\d z\d y\right) \\
&\qquad-\frac{\partial^3}{\partial t^3}\left(\int_{B_t(x)}\int_{B_{t-|y-x|}(y)}\frac1{16\pi^2|y-x||z-y|}\varepsilon\nabla\rho_1(y)\cdot\frac{z-y}{|z-y|}P_{r,\theta}^{(0)}(z)\d z\d y\right) \\
&\qquad-\frac{\partial^2}{\partial t^2}\left(\int_{B_t(x)}\int_{B_{t-|y-x|}(y)}\frac1{16\pi^2|y-x||z-y|^2}\varepsilon\nabla\rho_1(y)\cdot\frac{z-y}{|z-y|}P_{r,\theta}^{(0)}(z)\d z\d y\right).
\end{split}
\end{equation}
Remarking that we can explicitly calculate the primitive functions with respect to time of these integrals, for example:
\begin{align*}
\int_0^t\int_{B_\tau(x)}\int_{B_{\tau-|y-x|}(y)}F(x,y,z)\d z\d y\d\tau &= \int_{\R^3}\int_{\R^3}F(x,y,z)\int_0^t\chi_{[0,\infty)}(\tau-|y-x|-|z-y|)\d\tau\d z\d y \\
&= \int_{B_t(x)}\int_{B_{t-|y-x|}(y)}(t-|y-x|-|z-y|)F(x,y,z)\d z\d y,
\end{align*}
we integrate the result four times with respect to time to get rid of the time derivatives and then interchange the order of integration. This yields the expression
\[ P_{r,\theta,0}(t,x)+\varepsilon P_{r,\theta,1}(t,x) = \frac{\partial^4}{\partial t^4}\left(\int_{B_t(x)}K(t,x,z;\varepsilon)P_{r,\theta}^{(0)}(z)\d z\right) \]
with the kernel $K$ given by \eqref{eqFirSolWaveKernel}.


For the estimate of the higher order term we refer to \cite[Section 23.2, Lemma 23.2.1]{Hoe07}. Note that the coefficients of the homogeneous part of equation \eqref{eqFirP2} depend continuously on $\varepsilon$. In the limit case $\varepsilon=0$, equation \eqref{eqFirP2} results in the non-homogeneous wave equation with constant sound speed.
\end{proof}

\section{Properties of the Kernel}
\label{sec:kernel}
We want to study in this section the kernel $K$ defined in \eqref{eqFirSolWaveKernel}. In particular, we are interested in the behaviour of $K(t,x,y;\varepsilon)$ when $t$ gets much larger than the support of the involved functions $\alpha_1$ and $\rho_1$ and in the limit when $t$ approaches the distance $|y-x|$ from above (the kernel vanishes for $t<|y-x|$).

For large values of $t$, the change in the domain $\mathscr E_t(x,y)$ as $t$ varies does not change the value of the integral and $K$ behaves as a quadratic polynomial in $t$. Explicitly, we get with this argument the following expansion in $t$.
\begin{lemma}\label{thKerLargeTime}
Let $\alpha_1\in C_{\mathrm c}(\R^3)$, $\rho_1\in C_{\mathrm c}^1(\R^3)$, and $K$ be defined by \eqref{eqFirSolWaveKernel}. For arbitrary values $x,y\in\R^3$, let $T_{x,y}>0$ be chosen so that $\mathscr E_{T_{x,y}}(x,y)\supset\supp\rho_1\cup\supp \alpha_1$.

Then, we have for all $t\ge T_{x,y}$ the relation
\begin{equation}\label{eqKerLargeTime}
\begin{split}
K(t,x,y;\varepsilon) &= \left(\frac{1-\varepsilon \alpha_1(y)}{8\pi|y-x|}+\int_{\R^3}\frac1{32\pi^2|z-x||z-y|^2}\varepsilon\nabla\rho_1(z)\cdot\frac{z-y}{|z-y|}\d z\right)t^2 \\
&\qquad-\frac{1-\varepsilon(\alpha_1(y)+\rho_1(y))}{4\pi}\,t+\frac{|y-x|(1-\varepsilon \alpha_1(y))}{8\pi} \\
&\qquad+\int_{\R^3}\frac1{16\pi^2|z-x||z-y|}\varepsilon \alpha_1(z)\d z+\int_{\R^3}\frac{|z-x|^2-|z-y|^2}{32\pi^2|z-x||z-y|^2}\varepsilon\nabla\rho_1(z)\cdot\frac{z-y}{|z-y|}\d z.
\end{split}
\end{equation}
\end{lemma}
\begin{proof}
We replace in \eqref{eqFirSolWaveKernel} the integration over $\mathscr E_t(x,y)$ with the integral over $\R^3$, which does not change the value of the integral, since we have $\supp\rho_1\cup\supp \alpha_1\subset\mathscr E_{T_{x,y}}(x,y)\subset\mathscr E_t(x,y)$, and reorder the terms as coefficients of a polynomial in $t$. This yields
\begin{equation}\label{eqKerLargeTimePol}
\begin{split}
K(t,x,y;\varepsilon) &= \left(\frac{1-\varepsilon \alpha_1(y)}{8\pi|y-x|}+\int_{\R^3}\frac1{32\pi^2|z-x||z-y|^2}\varepsilon\nabla\rho_1(z)\cdot\frac{z-y}{|z-y|}\d z\right)t^2 \\
&\qquad-\left(\frac{1-\varepsilon \alpha_1(y)}{4\pi}+\int_{\R^3}\frac1{16\pi^2|z-y|^2}\varepsilon\nabla\rho_1(z)\cdot\frac{z-y}{|z-y|}\d z\right)t \\
&\qquad+\frac{|y-x|(1-\varepsilon \alpha_1(y))}{8\pi}+\int_{\R^3}\frac1{16\pi^2|z-x||z-y|}\varepsilon \alpha_1(z)\d z \\
&\qquad+\int_{\R^3}\frac{|z-x|^2-|z-y|^2}{32\pi^2|z-x||z-y|^2}\varepsilon\nabla\rho_1(z)\cdot\frac{z-y}{|z-y|}\d z.
\end{split}
\end{equation}

To simplify the coefficient of the linear term, we use the divergence theorem and that $\frac1{4\pi|x|}$ is the fundamental solution for the negative Laplace opertor $-\Delta$ and find
\[ \frac\varepsilon{16\pi^2}\int_{\R^3}\nabla\rho_1(z)\cdot\frac{z-y}{|z-y|^3}\d z = \frac\varepsilon{4\pi}\int_{\R^3}\frac{\Delta\rho_1(z)}{4\pi|z-y|}\d z = -\frac\varepsilon{4\pi}\rho_1(y). \]
Plugging this into \eqref{eqKerLargeTimePol}, we arrive at \eqref{eqKerLargeTime}.
\end{proof}

The other limit, where the integral over the spheroid $\mathscr E_t(x,y)$ simplifies is $t\downarrow|y-x|$. In this case, the spheroid shrinks to the line from $x$ to $y$.

\begin{lemma}\label{thIntLimit}
Let $x,y\in\R^3$ be two arbitrary, different points and $\psi\in C^2(\R^3)$. Then, the function $F:[|y-x|,\infty)\to\R$ defined by
\begin{equation}\label{eqKerTravelTimeCoorF}
F(t) = \int_{\mathscr E_t(x,y)}\frac{\psi(z)}{|z-x||z-y|}\d z,
\end{equation}
with $\mathscr E_t(x,y)=\{z\in\R^3\mid |z-x|+|z-y|\leq t\}$ as before, is two times differentiable and can be expanded around $t=|y-x|$ as
\[ F(t) = \frac{2\pi(t-|y-x|)}{|y-x|}\int_{L_{x,y}}\psi(z)\d s(z)+\frac{\pi(t-|y-x|)^2}{2|y-x|}\int_{L_{x,y}}\left(1-\frac{|z-x|}{|y-x|}\right)|z-x|\Delta\psi(z)\d s(z)+o((t-|y-x|)^2). \]
\end{lemma}

\begin{proof}
We introduce the new coordinates $r_1\ge0$, $r_2\ge0$, and $\varphi\in[0,2\pi]$ of a point $\phi(r_1,r_2,\varphi)\in\R^3$ such that
\[ r_1=|\phi(r_1,r_2,\varphi)-x|\quad\text{and}\quad r_2=|\phi(r_1,r_2,\varphi)-y|\quad\text{for every}\quad\varphi\in[0,2\pi]. \]
To this end, let $e_1=\frac{y-x}{|y-x|}$, $e_2$, $e_3$ be an orthonormal basis of $\R^3$ and set
\[ \phi(r_1,r_2,\varphi) = \frac{x+y}2+\xi(r_1,r_2)e_1+\eta(r_1,r_2)\cos\varphi\,e_2+\eta(r_1,r_2)\sin\varphi\,e_3 \]
with the functions $\xi$ and $\eta$ explicitly given by
\begin{equation}\label{eqKerTravelTimeCoorFcts}
\xi(r_1,r_2) = \frac{r_1^2-r_2^2}{2|y-x|}\quad\text{and}\quad\eta(r_1,r_2) = \sqrt{r_1^2-\left(\frac12|y-x|+\xi(r_1,r_2)\right)^2},
\end{equation}
see \autoref{fgKerTravelTime}.

\begin{figure}[ht]
\begin{center}
\begin{tikzpicture}[scale=2,font=\footnotesize]
\draw(-1.5,0)--(1.5,0);
\draw(-1,0.05)--(-1,-0.05) node[below] {$x$};
\draw(1,0.05)--(1,-0.05) node[below] {$y$};
\draw(0,0.05)--(0,-0.05) node[below] {$\tfrac12(x+y)$};
\draw(-1,0)--(0.6,0.9) node[pos=0.5,above] {$r_1$};
\draw(1,0)--(0.6,0.9) node[pos=0.5,right] {$r_2$};
\draw[->,thick](0,0)--(0.6,0) node[pos=0.5,above] {$\xi e_1$};
\draw[->,thick](0.6,0)--(0.6,0.9) node[pos=0.5,left] {$\eta e_2$};
\end{tikzpicture}
\end{center}
\caption{The relation between the variables $r_1$ and $r_2$ and the functions $\xi$ and $\eta$ drawn for $\varphi=0$. The explicit expressions \eqref{eqKerTravelTimeCoorFcts} for the functions $\xi$ and $\eta$ are obtained by solving the equation system of the two Pythagorean equations $r_1^2=(\tfrac12|x-y|+\xi)^2+\eta^2$ and $r_2^2=(\tfrac12|x-y|-\xi)^2+\eta^2$ for $\xi$ and $\eta$.}\label{fgKerTravelTime}
\end{figure}
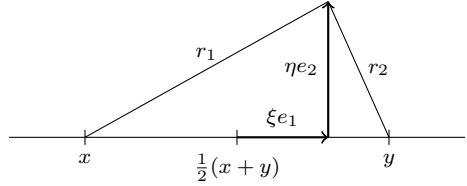

The volume element by switching to the new coordinates $r_1$, $r_2$, $\varphi$ is given by
\begin{align*}
|\det\d\phi(r_1,r_2,\varphi)| &= \left|\det
\begin{pmatrix}
\partial_{r_1}\xi(r_1,r_2) & \partial_{r_2}\xi(r_1,r_2) & 0 \\
\partial_{r_1}\eta(r_1,r_2)\cos\varphi & \partial_{r_2}\eta(r_1,r_2)\cos\varphi & -\eta(r_1,r_2)\sin\varphi \\
\partial_{r_1}\eta(r_1,r_2)\sin\varphi & \partial_{r_2}\eta(r_1,r_2)\sin\varphi & \phantom-\eta(r_1,r_2)\cos\varphi
\end{pmatrix}\right| \\
&= |\eta(r_1,r_2)(\partial_{r_2}\eta(r_1,r_2)\partial_{r_1}\xi(r_1,r_2)-\partial_{r_1}\eta(r_1,r_2)\partial_{r_2}\xi(r_1,r_2))|.
\end{align*}
Using the relations
\begin{align*}
\partial_{r_1}\eta(r_1,r_2) &= \phantom-\frac1{\eta(r_1,r_2)}\left(r_1-\left(\frac12|x-y|+\xi(r_1,r_2)\right)\partial_{r_1}\xi(r_1,r_2)\right), \\
\partial_{r_2}\eta(r_1,r_2) &= -\frac1{\eta(r_1,r_2)}\left(\frac12|x-y|+\xi(r_1,r_2)\right)\partial_{r_2}\xi(r_1,r_2),
\end{align*}
following directly from \eqref{eqKerTravelTimeCoorFcts}, we find that
\[ |\det\d\phi(r_1,r_2,\varphi)| = |r_1\partial_{r_2}\xi(r_1,r_2)| = \frac{r_1r_2}{|y-x|}. \]

Thus, switching to these new coordinates (and remembering that a point with distance $r_1$ to $x$ and $r_2$ to $y$ only exists if the triangle inequalities $||y-x|-r_1|\le r_2\le|y-x|+r_1$ are fulfilled), the integral in \eqref{eqKerTravelTimeCoorF} becomes
\[ F(t) = \frac1{|y-x|}\int_0^{\frac{t+|y-x|}{2}}\int_{||y-x|-r_1|}^{\min\{t-r_1,|x-y|+r_1\}}\int_0^{2\pi}\psi(\phi(r_1,r_2,\varphi))\d\varphi\d r_2\d r_1. \]

To expand $F$ now in a Taylor polynomial around $t=|y-x|$, we first realise that $F(|y-x|) = 0$ and then calculate the derivatives at $t=|y-x|$. For the first derivative, we find (the inequality $||y-x|-r_1|\le t-r_1\le|x-y|+r_1$ is equivalent to $\frac{t-|y-x|}2\le r_1\le\frac{t+|y-x|}2$) that
\[ F'(t) = \frac1{|y-x|}\int_{\frac{t-|y-x|}2}^{\frac{t+|y-x|}2}\int_0^{2\pi}\psi(\phi(r_1,t-r_1,\varphi))\d\varphi\d r_1. \]
Since $\phi(r_1,|y-x|-r_1,\varphi)=x+r_1e_1$ for all $r_1\in[0,|y-x|]$ and all $\varphi\in[0,2\pi]$, we get
\[ F'(|y-x|) = \frac{2\pi}{|y-x|}\int_0^{|y-x|}\psi\left(x+r_1\frac{y-x}{|y-x|}\right)\d r_1 = \frac{2\pi}{|y-x|}\int_{L_{x,y}}\psi(z)\d s(z), \]
where $L_{x,y}$ denotes the straight line between the points $x$ and $y$.

For the second derivative, we obtain
\begin{multline*}
F''(t) = \frac1{2|y-x|}\int_0^{2\pi}\left(\psi\left(\phi\left(\frac{t+|y-x|}2,\frac{t-|y-x|}2,\varphi\right)\right)-\psi\left(\phi\left(\frac{t-|y-x|}2,\frac{t+|y-x|}2,\varphi\right)\right)\right)\d\varphi \\
+\frac1{|y-x|}\int_{\frac{t-|y-x|}2}^{\frac{t+|y-x|}2}\int_0^{2\pi}\nabla\psi(\phi(r_1,t-r_1,\varphi))\cdot\partial_{r_2}\phi(r_1,t-r_1,\varphi)\d\varphi\d r_1.
\end{multline*}
To calculate the limit $t\downarrow|y-x|$, we remark that
\begin{multline*}
\int_0^{2\pi}\nabla\psi(\phi(r_1,t-r_1,\varphi))\cdot\partial_{r_2}\phi(r_1,t-r_1,\varphi)\d\varphi \\
= \partial_{r_2}\xi(r_1,t-r_1)\int_0^{2\pi}\nabla\psi(\phi(r_1,t-r_1,\varphi))\cdot e_1\d\varphi+\frac{\partial_{r_2}\eta(r_1,t-r_1)}{\eta(r_1,t-r_1)}\int_{\partial D(t,r_1)}\nabla\psi(z)\cdot\nu(z)\d s(z),
\end{multline*}
where we introduced the manifold $D(t,r_1)$ with boundary given as the disc normal to $e_1$ with center in $\frac12(x+y)+\xi(r_1,t-r_1)e_1$ and radius $\eta(r_1,t-r_1)$, whose boundary $\partial D(t,r_1)$ is parametrised by $\phi(r_1,t-r_1,\cdot)$ and its unit normal vector field $\nu$ is given by $\nu(\phi(r_1,t-r_1,\varphi))=\cos\varphi e_2+\sin\varphi e_3$. Now, using Stokes' theorem, we can rewrite the boundary integral to an integral over the two-dimensional disc and find
\[ \int_{\partial D(t,r_1)}\nabla\psi(z)\cdot\nu(z)\d s(z) = \int_{D(t,r_1)}(\Delta\psi(z)-(e_1\cdot\nabla)^2\psi(z))\d s(z). \]
Plugging all this in our formula for $F''$, we can take the limit $t\downarrow|y-x|$ and get with
\[ \partial_{r_2}\xi(r_1,|y-x|-r_1) = -\left(1-\frac{r_1}{|y-x|}\right)\quad\text{and}\quad\lim_{t\downarrow|y-x|}\eta(r_1,t-r_1)\partial_{r_2}\eta(r_1,t-r_1) = r_1\left(1-\frac{r_1}{|y-x|}\right) \]
that
\begin{multline*}
F''(|y-x|) = \frac{\pi(\psi(y)-\psi(x))}{|y-x|} \\
+\frac{2\pi}{|y-x|}\int_0^{|y-x|}\left(1-\frac{r_1}{|y-x|}\right)\left(\frac{r_1}2\left(\Delta\psi(x+r_1e_1)-(e_1\cdot\nabla)^2\psi(x+r_1e_1)\right)-e_1\cdot\nabla\psi(x+r_1e_1)\right)\d r_1.
\end{multline*}
Since an integration by parts gives us
\begin{align*}
\int_0^{|y-x|}\left(1-\frac{r_1}{|y-x|}\right)\frac{r_1}2(e_1\cdot\nabla)^2\psi(x+r_1e_1)\d r_1 &= -\int_0^{|y-x|}\left(\frac12-\frac{r_1}{|y-x|}\right)e_1\cdot\nabla\psi(x+r_1e_1)\d r_1 \\
&= \int_0^{|y-x|}\frac{r_1}{|y-x|}e_1\cdot\nabla\psi(x+r_1e_1)\d r_1-\frac12(\psi(y)-\psi(x)),
\end{align*}
all but the first term in the integral cancel each other, and the expression simplifies to
\begin{align*}
F''(|y-x|) &= \frac\pi{|y-x|}\int_0^{|y-x|}\left(1-\frac{r_1}{|y-x|}\right)r_1\Delta\psi(x+r_1e_1)\d r_1 \\
&= \frac\pi{|y-x|}\int_{L_{x,y}}\left(1-\frac{|z-x|}{|y-x|}\right)|z-x|\Delta\psi(z)\d s(z).
\end{align*}
\end{proof}

\begin{corollary}
 Let $\alpha_1\in C^2(\R^3)$, $\rho_1\in C^1(\R^3)$ and $K(t,x,y;\varepsilon)$ be given by \eqref{eqFirSolWaveKernel}. Then
 \begin{equation}\label{eq:KLimit}
  \lim_{t\downarrow|y-x|} \frac{\partial}{\partial t} K(t,x,y;\varepsilon) = \frac{1}{8\pi|y-x|}\int_{L_{x,y}}\varepsilon \alpha_1(z)\d s(z)
 \end{equation}
 and
 \begin{equation}
  \lim_{t\downarrow|y-x|} \frac{\partial^2}{\partial t^2} K(t,x,y;\varepsilon) = \frac{1-\varepsilon\alpha_1(y)}{4\pi|y-x|} + \frac{\varepsilon\rho_1(x) - \varepsilon\rho_1(y)}{8\pi|y-x|} + \frac{1}{16\pi|y-x|}\int_{L_{x,y}}\left(1-\frac{|z-x|}{|y-x|}\right)|z-x|\varepsilon\Delta \alpha_1(z)\d s(z).
 \end{equation}
\end{corollary}
\begin{proof}
 The first identity follows immediately by applying the previous lemma on the kernel. For the second identity, note that
 \[
  \int_{L_{x,y}} \nabla\rho_1(z)\cdot\frac{z-y}{|z-y|}\d s = -\int_0^{|y-x|} \frac{\d}{\d \lambda} \rho_1\left(x+\lambda\frac{y-x}{|y-x|}\right) \d \lambda = \rho_1(x)-\rho_1(y).
 \]

\end{proof}

\section{Forward Problem}
\label{sec:dp}

We will now consider the idealised problem where the light illumination $\Phi_{r,\theta}$ of the laser pulse is perfectly focused on the plane $E_{r,\theta}$ and where the perturbations of the parameters $\rho$ and $\alpha$ are so small that we can ignore the error term $P_{r,\theta,2}$ for the measurements. 
Henceforth we shall assume that we are given the measurements
\[ M_{r,\theta}(t,x) = \int_{B_t(x)\cap E_{r,\theta}}K(t,x,y;\varepsilon)f(y)\d y = \mathcal{R}_3[K(t,x,.;\varepsilon)f(.)](r,\theta)\]
with $f(y)=\Phi^{(0)}\gamma(y)\mu(y)$ and $\mathcal{R}_3$ being the three-dimensional Radon transform. Note that the kernel vanishes for $|y-x|>t$.

For the following discussion, we shall assume that we are given the measurement data $M_{r,\theta}(t,x)$ for all 
$r\geq0$, $\theta \in {\mathbf S}^2$, and $t>0$ at distinct points $x \in \Sigma$, where $\Sigma$ is a detector surface enclosing the compact support of $f$. We assume the seeked functions to be 
sufficiently smooth and that the perturbation of the sound speed and density are supported in a strict subset of $\interior(\supp f)$.


\tdplotsetmaincoords{70}{105}  
\pgfmathsetmacro{\inradius}{1.5} 
\pgfmathsetmacro{\outradius}{1.7}
\pgfmathsetmacro{\detectorradius}{5.2}
\pgfmathsetmacro{\detectorRout}{5.5}
\pgfmathsetmacro{\detectorRin}{4.9}

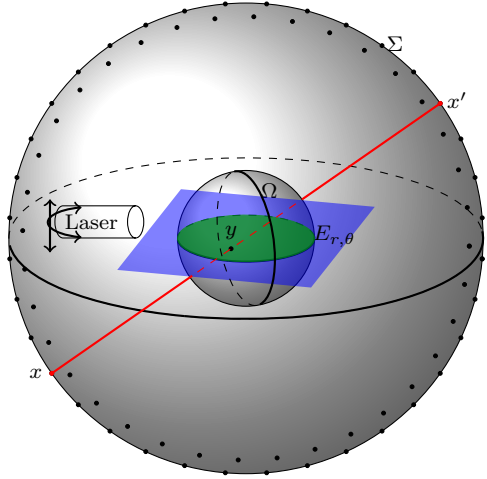
\begin{figure}[ht]
\centering
\begin{tikzpicture}[font=\footnotesize,scale=0.6,tdplot_main_coords]


\shadedraw[tdplot_screen_coords,ball color = white] (0,0) circle (\detectorradius);
\shadedraw[tdplot_screen_coords,ball color = white] (0,0) circle (\inradius);

\coordinate[label=below:$\Omega$] (OM) at(70:1.5cm);

\newcommand \drawingArcs[1]{%
  \draw[dashed] (#1,0,0) arc (0:360:#1);
  \draw[thick] (#1,0,0) arc (0:110:#1);
  \draw[thick] (#1,0,0) arc (0:-70:#1);
}
\drawingArcs{\detectorradius};
\drawingArcs{\inradius};

\node[cylinder, draw, shape aspect=.5, alias=cyl](laser) at (-0.15,-3.55) {Laser};
\draw[thick,<->] (-1.0,-4.0,0) arc (-150:0:1);
\draw[thick,<->] (0,-4.45,-0.75) -- (0,-4.45,0.5) ;

\fill[opacity=0.5,blue] (\inradius+1.2,\inradius+0.7,0) -- (\inradius+1.2,-\inradius-0.7,0) -- (-\inradius-1.2,-\inradius-0.7,0) -- (-\inradius-1.2,\inradius+0.7,0) -- cycle;
\coordinate[label=below:$E_{r,\theta}$] (plane) at (14:2.cm);

\filldraw[opacity=0.7,black!40!green] (0,0) circle (\inradius);

\tdplotsetthetaplanecoords{40}
\draw[thick,tdplot_rotated_coords] (\inradius,0,0) arc (0:151:\inradius);
\draw[dashed,tdplot_rotated_coords] (\inradius,0,0) arc (180:-40:-\inradius);
\draw[thick,tdplot_rotated_coords] (\inradius,0,0) arc (360:336:\inradius);

\foreach \yy in {0,...,36}
{
	\pgfmathsetmacro{\ang}{10*\yy-5}
	\filldraw[black] (\ang:\detectorradius cm) circle(1.2pt);
	\filldraw[black] (\ang+3:\detectorRin cm) circle(1.2pt);
}
\coordinate[label=below:$\Sigma$] (sig) at (55:5.7cm);

\filldraw[black] (215:0.4cm) circle(1.2pt);
\coordinate[label=above:$y$] (y) at (215:0.4cm);

\filldraw[red] (215:\detectorradius cm) circle(1.2pt);
\coordinate[label=left:$x$] (x) at (215:\detectorradius cm);

\draw[thick,red] (215:\detectorradius cm) -- (215:1.44cm);

\draw[dashed,red] (215:-\detectorradius cm) -- (215:1.38cm);

\draw[thick,red] (35:1.5 cm) -- (35:\detectorradius cm);

\filldraw[red] (35:\detectorradius cm) circle(1.2pt);
\coordinate[label=right:$x'$] (xprime) at (36:\detectorradius cm);

\end{tikzpicture}

\caption{Measurement setup }\label{fig:1}
\end{figure}


\section{Reconstruction Algorithm}
\label{sec:reconstruction}
Under the assumptions of the previous section, we are going to present a reconstrucion method for the functions $f$, $\alpha_1$ and $\rho_1$. Given our measurement data, we define the functions
\begin{align}
 M_{0,x}(y) &= \lim_{t\downarrow|y-x|}\partial_t\mathcal{R}_3^{-1}[M_{r,\theta}(t,x)](y),\nonumber\\
 N_{0,x}(y) &= \lim_{t\downarrow|y-x|}\partial_{tt}\mathcal{R}_3^{-1}[M_{r,\theta}(t,x)](y),\nonumber\\
 M_{\infty,x}(y) &= \lim_{t\rightarrow\infty} \partial_{tt}\mathcal{R}_3^{-1}[M_{r,\theta}(t,x)](y),\nonumber\\
 N_{\infty,x}(y) &= \lim_{t\rightarrow\infty} \left( M_{\infty,x}(y)t-\partial_t\mathcal{R}_3^{-1}[M_{r,\theta}(t,x)](y)\right),\nonumber
\end{align}
where $x\in\Sigma$ and $y\in\mathbb{R}^3$. Then, by virtue of \autoref{thKerLargeTime} and \autoref{thIntLimit} together with integration by parts, we obtain
\begin{align}\label{eq:MZero}
 M_{0,x}(y) &= \frac{f(y)}{8\pi|y-x|}\int_{L_{x,y}}\varepsilon \alpha_1(z)\d s(z),\\
\label{eq:NZero}
 N_{0,x}(y) &= f(y)\left(\frac{1-\varepsilon\alpha_1(y)}{4\pi|y-x|}-\frac{\varepsilon\rho_1(y)}{8\pi|y-x|} + \frac{1}{16\pi|y-x|}\int_{L_{x,y}}\left(1-\frac{|z-x|}{|y-x|}\right)|z-x|\varepsilon\Delta \alpha_1(z)\d s(z)\right),\\
 M_{\infty,x}(y) &= f(y)\left( \frac{1-\varepsilon \alpha_1(y)- \varepsilon \rho_1(y)}{4\pi|y-x|}+\int_{\R^3}\frac{\varepsilon\rho_1(z)}{16\pi^2}\nabla_z\left(\frac{1}{|z-x|}\right)\cdot\nabla_z\left(\frac{1}{|z-y|}\right)\d z\right),\\
\label{eq:NInf}
 N_{\infty,x}(y) &= f(y)\frac{1-\varepsilon \alpha_1(y) - \varepsilon \rho_1(y)}{4\pi}.
\end{align}
In the first step, we are going to reconstruct $\alpha_1$. Equation \eqref{eq:NInf} gives us the values of $f$ on $\supp f \setminus (\supp \alpha_1 \cup \supp\rho_1)$. By knowledge of $f$ in a vicinity close to the boundary of $\supp f$ together with equation \eqref{eq:MZero}, we infer to know all line integrals
of $\alpha_1$.
\begin{proposition}\label{Prop_Reconstruct_q}
Assume $\supp\alpha_1\cup\supp\rho_1\subsetneq \interior(\supp f)$ and suppose that the data $M_{0,x}$ and $N_{\infty,x}$ is given for all $x\in\Sigma$. Then the three-dimensional $X$-ray transform $X_3[\varepsilon \alpha_1](y,v)$ is known for all $y\in\mathbb{R}^3$ and $v\in\mathbf{S}^2$.
\end{proposition}
\begin{proof}
Let $[0,\infty) \ni \lambda \mapsto \gamma_v(\lambda)= x+\lambda v$, $v=(y-x)/|y-x|$, be the parametrisation of the half-ray connecting $x\in\Sigma$ with a point $y\in\supp f$. Thus, 
$y=\gamma_v(\mu)$ for some $\mu>0$ and the line integral in \eqref{eq:MZero} becomes 
\[
 \int_{L_{x,y}}\varepsilon \alpha_1(z)\d s(z) = \int_0^{\mu} \varepsilon \alpha_1(\gamma_v(\lambda))\d \lambda. 
\]
Choose $x'\in\Sigma$ such that $y\in L_{x,x'}$. Then
\[
 8\pi\left(M_{0,x}(y)|y-x| + M_{0,x'}(y)|y-x'|\right) = f(y)\int_{L_{x,x'}}\varepsilon \alpha_1(z)\d s(z).
\]
Let $y_0\in L_{x,x'}\cap\left(\supp f\setminus \left(\supp\alpha_1\cup\supp\rho_1\right)\right)$ for which $f(y_0)\neq 0$ is known. Division through $f(y_0)$ gives us
\[
 X_3[\varepsilon\alpha_1](y,v) = 8\pi\frac{M_{0,x}(y_0)|y_0-x| + M_{0,x'}(y_0)|y_0-x'|}{f(y_0)}.
\]
\end{proof}
By inverting the $X$-ray transform, we obtain $\alpha_1$. 
\begin{proposition}
 For $\alpha_1$ sufficiently small in the $C^2$-norm we get
 \begin{align}
 f(y)&=\left( 8\pi|y-x|N_{0,x}(y) - 4\pi N_{\infty,x}(y) \right)\left(1-\varepsilon\alpha_1(y) + \frac{1}{2}\int_{L_{x,y}}\left(1-\frac{|z-x|}{|y-x|}\right)|z-x|\varepsilon\Delta \alpha_1(z)\d s(z)\right)^{-1}.\\
 \varepsilon\rho_1(y) &= 1 - \varepsilon\alpha_1(y) - \frac{4\pi N_{\infty,x}(y)}{f(y)}.
\end{align}
\end{proposition}
\begin{proof}
Combining \eqref{eq:NZero} with \eqref{eq:NInf}, we get
\[
 8\pi|y-x|N_{0,x}(y) - 4\pi N_{\infty,x}(y)=f(y)\left(1-\varepsilon\alpha_1(y) + \frac{1}{2}\int_{L_{x,y}}\left(1-\frac{|z-x|}{|y-x|}\right)|z-x|\varepsilon\Delta \alpha_1(z)\d s(z)\right)
\]
and the first identity follows. The distortion of the density can be expressed by virtue of \eqref{eq:NInf}.
\end{proof}

\section*{Acknowledgement}
The work of AB and OS has been supported by the Austrian Science Fund (FWF), Project P26687-N25 (Interdisciplinary Coupled Physics Imaging).

\section*{References}
\renewcommand{\i}{\ii}
\printbibliography[heading=none]


\end{document}